\documentclass[11pt,twoside]{amsart}

\usepackage{latexsym}
\usepackage{amssymb}
\usepackage{vmargin}

\usepackage{stmaryrd}
\usepackage{euscript}
\usepackage{mathrsfs}
\usepackage[all]{xy}

\DeclareMathAlphabet{\mathpzc}{OT1}{pzc}{m}{it}
\usepackage{xr-hyper}
\usepackage{hyperref}
\setmargins{32mm}{20mm}{14.6cm}{22cm}{1cm}{1cm}{1cm}{1cm}

\newcommand\la{\leftarrow}

\newcommand\id{\mathrm{id}}

\newcommand\ten{\otimes}

\newcommand\eps{\epsilon}

\newcommand\DD{\mathrm{D}}

\renewcommand\H{\mathrm{H}}
\newcommand\z{\mathrm{Z}}
\renewcommand\b{\mathrm{B}}

\newcommand\Q{\mathbb{Q}}

\newcommand\C{\mathcal{C}}

\newcommand\cL{\mathcal{L}}

\newcommand\cP{\mathcal{P}}

\newcommand\cS{\mathcal{S}}

\newcommand\cU{\mathcal{U}}

\newcommand\cW{\mathcal{W}}

\newcommand\Del{\mathfrak{Del}}

\newcommand\DDel{\uline{\Del}}

\newcommand\m{\mathfrak{m}}

\newcommand\Ho{\mathrm{Ho}}

\newcommand\Alg{\mathrm{Alg}}

\newcommand\Hom{\mathrm{Hom}}

\newcommand\map{\mathrm{map}}

\newcommand\HHom{\underline{\mathrm{Hom}}}
\newcommand\EEnd{\underline{\mathrm{End}}}
\newcommand\DDer{\underline{\mathrm{Der}}}

\newcommand\coker{\mathrm{coker\,}}

\newcommand\Co{\mathrm{Co}}

\newcommand\nilp{\mathrm{nilp}}

\newcommand\Art{\mathrm{Art}}

\newcommand\Set{\mathrm{Set}}

\newcommand\Com{\mathrm{Com}}

\newcommand\Lim{\varprojlim}
\newcommand\LLim{\varinjlim}

\newcommand\into{\hookrightarrow}
\newcommand\onto{\twoheadrightarrow}

\newcommand\xra{\xrightarrow}
\newcommand\xla{\xleftarrow}

\newcommand\bt{\bullet}
\newcommand\by{\times}

\newcommand\mc{\mathrm{MC}}
\newcommand\mmc{\underline{\mathrm{MC}}}

\newcommand\Gg{\mathrm{Gg}}

\newcommand\ddef{\mathrm{Def}}

\newcommand\Symm{\mathrm{Symm}}
\newcommand\Sym{\mathrm{Sym}}

\newcommand\Tot{\mathrm{Tot}\,}
\newcommand\diag{\mathrm{diag}\,}

\newcommand\pro{\mathrm{pro}}

\newcommand\half{\frac{1}{2}}

\newcommand\gr{\mathrm{gr}}

\newcommand\Gpd{\mathrm{Gpd}}

\newcommand\Lie{\mathrm{Lie}}

\newcommand\op{\mathrm{opp}}

\newcommand\co{\colon\thinspace}

\newcommand\oR{\mathbf{R}}

\newcommand\oL{\mathbf{L}}

\newcommand\uleft\underleftarrow
\newcommand\uline\underline
\newcommand\uright\underrightarrow

\newcommand{\tps}{\texorpdfstring}

\newtheorem{theorem}{Theorem}[section]
\newtheorem{proposition}[theorem]{Proposition}
\newtheorem{corollary}[theorem]{Corollary}

\newtheorem{lemma}[theorem]{Lemma}
\newtheorem*{theorem*}{Theorem}
\newtheorem*{proposition*}{Proposition}
\newtheorem*{corollary*}{Corollary}
\newtheorem*{lemma*}{Lemma}
\newtheorem*{conjecture*}{Conjecture}

\theoremstyle{definition}
\newtheorem{definition}[theorem]{Definition}
\newtheorem*{definition*}{Definition}

\newtheorem*{notation*}{Notation}

\theoremstyle{remark}
\newtheorem{example}[theorem]{Example}

\newtheorem{remark}[theorem]{Remark}
\newtheorem{remarks}[theorem]{Remarks}

\newtheorem*{example*}{Example}
\newtheorem*{examples*}{Examples}
\newtheorem*{remark*}{Remark}
\newtheorem*{remarks*}{Remarks}
\newtheorem*{exercise*}{Exercise}
\newtheorem*{property*}{Property}
\newtheorem*{properties*}{Properties}

\newcommand\lex{\mathrm{lex}}
\newcommand\DGLA{\mathrm{DGLA}}

\begin{document}


%
%
%

\begin{abstract}
 We summarise the chain of comparisons \cite{ddt1} showing Hinich's  derived Maurer--Cartan functor gives an equivalence between differential graded Lie algebras and derived Schlessinger functors on  Artinian differential graded-commutative algebras. We include some motivating deformation problems and analogues for more general Koszul dual pairs of operads. 
\end{abstract}

 \title{Derived deformation functors, Koszul duality, and Maurer-Cartan spaces}
 \author{J.P.Pridham}
\maketitle

\tableofcontents


\section{Derived deformation functors}

Where classical deformation theory concerns functors on local Artinian rings, derived deformation theory looks at functors on enhancements such as  differential graded or simplicial rings. We will focus on  the former, which only apply in characteristic $0$. Fix a field $k$ of characteristic $0$.

\subsection{Artinian cdgas and DGLAs} 

\begin{definition}\label{cdgadef}
A cdga (commutative differential graded algebra) $A_{\bt}$ over $k$ is a chain complex of $k$-vector spaces equipped with a   unital associative graded-commutative  multiplication with respect to which the differential acts as a derivation. Say that $A_{\bt}$ is local Artinian if it admits a $k$-cdga homomorphism $A_{\bt} \to k$ for which the kernel $\m(A_{\bt})$ is nilpotent and finite-dimensional. 

Write $dg\Art_k$ for the category of local Artinian $k$-cdgas, $dg_+\Art_k \subset dg\Art$ for the full subcategory of non-negatively graded objects  $\ldots A_2 \to A_1 \to A_0$, and $\Art_k \subset dg_+\Art_k$ for the full subcategory on objects concentrated in degree $0$.
\end{definition}

\begin{definition}
 A differential graded Lie algebra (DGLA) $L^{\bt}$ over $k$ is a cochain complex  of $k$-vector spaces equipped with a graded-Lie bracket with respect to which the differential acts as a derivation.
\end{definition}



\begin{remark}\label{cfluriermk}
 As is traditional in derived deformation theory \cite{hinstack,Man2,ddt1}, the nilpotence condition in Definition \ref{cdgadef} is strict: there exists $n>0$ such that any product of $n$ elements in $\m(A_{\bt})$ is $0$. The formulation of \cite[\S 6.2]{lurie} and \cite{lurieDAG10} instead considers cdgas which are only homologically Artinian, but they break the Maurer--Cartan functor which will be our main focus, although their homotopy theory is equivalent
by \cite[Proposition 2.7]{drep} or \cite[Corollary 4.4.4]{boothDefThPt}.\footnote{
Specifically,  the Postnikov tower ensures  that homologically Artinian cdgas lie in the closure of $dg_+\Art_k$ under homotopy pullbacks, so \cite[Corollary 4.4.4]{boothDefThPt} implies that derived completion gives a fully faithful $\infty$-functor 
from them to the simplicial localisation $\oL^{\cW}\pro(dg_+\Art_k)$ 
of 
the pro-category 
at quasi-isomorphisms $\cW$. \\
If, for $A,B \in dg_+\Art_k$, we have a span $A \xla{p} A' \xra{q} B$ in $\pro(dg_+\Art_k)$ with $p$ a trivial fibration and $q$ a fibration, then \cite[Lemma 4.5]{ddt1}\footnotemark allows us to write $p$ as a filtered limit of a system of acyclic small extensions in $dg_+\Art$. Since $q$ must factor through some term $A''$ in the system,  we can replace $A'$ with its quotient $A'' \in dg_+\Art_k$ with the same properties. Thus  the essential image of $dg_+\Art_k \to \Ho(\pro(dg_+\Art_k))$ is closed under homotopy pullbacks. \\
%
Since  $\cW$-local equivalences (in the sense of left Bousfield localisation) are preserved under filtered colimits, that lemma also implies that the  Yoneda functor $h \co \pro(dg_+\Art_k)^{\op} \into (\Set)^{dg_+\Art_k} \into (s\Set)^{dg_+\Art_k}$ sends quasi-isomorphisms to $\cW$-local equivalences. If $\tilde{A}^{\bt}$ is a cofibrant cosimplicial frame  in $\pro(dg_+\Art_k)$ for $A \in dg_+\Art_k$, the map $h_A \to h_{\tilde{A},\bt}$ is thus a  $\cW$-local equivalence, and $h_{\tilde{A},\bt}$ is $\cW$-local. Combining \cite[\S 5.4]{hovey} with the proof of \cite[Theorem 2.2]{DKEquivsHtpyDiagrams}, it follows that $\oL^{\cW}dg_+\Art_k \to \oL^{\cW}\pro(dg_+\Art_k)$ is also a fully faithful $\infty$-functor.}
\footnotetext{Rather, we  partially generalise  to $\pro(dg_+\Art_k)$. Take a minimal dg ideal $I \le A'$ for which $A'/I \to A$ has the desired property; this must exist by pro-Artinianity. Then $\H_*I=0$ and any acyclic subcomplex of $I$ containing $\m(A) I$ would contradict minimality,  
so $(I/\m(A)I)_n\cong \H_n(I/\m(A)I) \cong \H_{n-1}(\m(A)I)$. Thus $I_{\le n-1}=0$ implies $I_{\le n}=0$, so $I=0$.}
\end{remark}

\subsection{Extended functors associated to DGLAs (after \tps{\cite{Man2, Kon} \ldots}{[Manetti,Kontsevich, ...]})}

\begin{definition}\label{MCdef}
 Given a DGLA $L$, the Maurer--Cartan set $\mc(L)$ is defined by 
 \[
  \mc(L):=\{\omega \in L^1 ~:~ d\omega +\half [\omega,\omega]=0\}.
 \]

 If $L^0$ is nilpotent, define the gauge group $\Gg(L)$ to consist of grouplike elements in the completed universal enveloping algebra $\hat{\cU}(L^0)$ (a complete Hopf algebra). The exponential map gives an isomorphism to $\Gg(L)$ from the set $L^0$ equipped with the Campbell--Baker--Hausdorff product.
 
 There is a gauge action of $\Gg(L)$ on $\mc(L)$, given by $g\star \omega:= g\omega g^{-1} -(dg)g^{-1}$ (evaluated in $\hat{\cU}(L)^1$); see \cite[\S 1]{Man} or \cite[Lecture 3]{Kon}.

 Denote the quotient set $\mc(L)/\Gg(L)$ by $\ddef(L)$, and the quotient groupoid $[\mc(L)/\Gg(L)]$ by $\Del(L)$ (the Deligne groupoid).
 \end{definition}
The terminology has its origins in constructions associated to the DGLA of differential forms valued in an adjoint bundle, where the Maurer--Cartan equation parametrises flat connections and the gauge action corresponds to gauge transformations.

This simple lemma is key to the role these functors play in deformation theory:
\begin{lemma}\label{obsdglalemma}
Given a  central extension $I \to L \to M$   of DGLAs, we have short exact sequences
\begin{align*}
 \z^1(I) \to \mc(L) \to \mc(M) \to \H^2(I)\\
 I^0 \to \Gg(L) \to \Gg(M) \to 0 \\
 \H^1(I) \to \ddef(L) \to \ddef(M) \to \H^2(I)
 \end{align*}
of groups and sets (where $L^0$ is assumed nilpotent for the last two sequences).
 \end{lemma}
\begin{proof}
 Given $\omega \in \mc(M)$, surjectivity of $L^1 \to M^1$ allows us to choose a lift $\tilde{\omega} \in L^1$, and then $\kappa(\tilde{\omega}):=d\tilde{\omega}+ \half [\tilde{\omega},\tilde{\omega}]$ lies in $\ker(L^2 \to M^2)=I^2$. Moreover, since $[\tilde{\omega},[\tilde{\omega},\tilde{\omega}]]=0$ and $d^2=0$,  we have
$
  d \kappa(\tilde{\omega})= [d\tilde{\omega},\tilde{\omega}] = [\kappa(\tilde{\omega}), \tilde{\omega}],
$
which vanishes because $I$ is central. Hence $\kappa(\tilde{\omega}) \in \z^2I$. Any other lift for $\omega$ takes the form $\tilde{\omega}+x$ for $x \in I^1$, with $\kappa(\tilde{\omega}+x)= \kappa(\tilde{\omega})+dx$ (centrality again). Thus  the class $[\kappa(\tilde{\omega})] \in \H^2I$ is the potential obstruction to lifting $\omega$ to $\mc(L)$, and the set of lifts is a torsor for $\z^1I$ under addition.
The exact sequence for $\Gg$ is immediate, and that for $\ddef$ then follows by passing to quotients.
 \end{proof}

Given a $k$-DGLA $L$ and $A \in dg\Art_k$, we have a nilpotent DGLA $\Tot(L\ten_k \m(A))$ given by 
\[
 \Tot(L\ten \m(A))^n= \bigoplus_i L^{n+i}\ten_k\m(A)_i 
\]
(finite sum because $A$ is Artinian), with bracket
$
 [u\ten a,v\ten b]= (-1)^{\deg v \deg a} [u,v]\ten ab
$
and differential
$
 d(u\ten a)= (du)\ten a +(-1)^{\deg u} u\ten da.
$

\begin{definition}
 Given a DGLA $L$, define set-valued functors $\mc(L,-)$ and $\ddef(L,-)$ and the groupoid-valued functor $\Del(L,-)$  on $dg\Art_k$ by evaluating the functors $\mc$, $\ddef$ and $\Del$ respectively on the nilpotent DGLA
 $\Tot(L\ten_k \m(A))$.
\end{definition}

%
%
%
%
\begin{definition}\label{epsndef}
 Define $k[\eps_n]\in  dg\Art_k$ to be $k \oplus k.\eps_n$, with $\deg \eps_n = n$ and $\eps_n\cdot \eps_n=0$. 
\end{definition}

\begin{example}\label{ddefcohoex}
We have isomorphisms 
 $\mc(L, k[\eps_n]) \cong \z^{n+1}(L)\eps_n$ and  $\Gg(L, k[\eps_n]) \cong L^n \eps_n $, with  
 \[
 \ddef(L, k[\eps_n]) \cong \coker(L^n \eps_n \xra{d} \z^{n+1}(L)\eps_n) = \H^{n+1}(L). 
\]

In particular, the functor $\ddef(L,-)$ on  $dg\Art_k$ detects all cohomology groups of $L$.
 \end{example}

\subsection{What do these classify?}


\begin{example} \label{defchainex}
 For $V_{\bt}$ a chain complex of $k$-vector spaces, we have a DGLA $L:= \EEnd(V)$, where $L^n= \prod_i \Hom_k(V_{i+n},V_i)$, with bracket $[f,g]= f \circ g - (-1)^{\deg f \deg g} g \circ f$ and differential $d(f):= [d,f]$.
 
 Extending $A$-linearly, elements of $(\Tot(L\ten \m(A)))^1$ can be identified with $A$-linear morphisms  $\Tot(A \ten V) \to \Tot(\m(A) \ten V)$ of homological degree $-1$. Such an element satisfies the Maurer--Cartan equation if and only if  the corresponding map $\omega \co  A\ten V \to  A\ten V[1]$ satisfies $(d + \omega) \circ (d + \omega)=0$.
 Meanwhile, 
\[
 \Gg(L,A) \cong \{g \in \HHom_A(A\ten V, A\ten V)_0 ~:~ g \equiv \id \mod \m(A)\},
\]
with the gauge action on $\mc(L,A)$ corresponding to conjugation.
 
Since flat modules over Artinian rings are free, one way to interpret $\ddef(L,A)$ is thus that it parametrises isomorphism classes of  $A$-modules $V'$ in chain complexes which are flat as graded $A$-modules, equipped with a fixed isomorphism $V'\ten_Ak \cong V$.

Beware that this flatness condition interacts poorly with quasi-isomorphism, and in particular 
$V'\ten_Ak$ 
does not necessarily compute
$V'\ten^{\oL}_Ak$ 
unless both $A$ and $V$ are concentrated in non-negative chain degrees (see \cite[Example 4.3]{hinichDefsHtpyAlg} for a counterexample).
It is instead a derived tensor product of the second kind as in \cite[\S 3.12]{positselskiDerivedCategories}; see \cite{coddt} for analysis of such deformations.
\end{example}
%
%
%

\begin{example}\label{defalgex}
 For algebras $R$ over $k$-linear dg operads $\cP$, we can similarly consider the DGLA $L:=\DDer_{\cP}(R,R)$ of derivations of $R$ as a graded $\cP$-algebra, and then $\mc(L,A)$ parametrises deformations of the structural derivation $d$ on $R$, i.e.
  closed derivations $d'$ of homological degree $-1$ on the  $\cP$-algebra $R\ten_kA$ in graded $A$-modules with $d' \equiv d \mod \m(A)$. 
  
 The gauge group $\Gg(A)$ then consists of automorphisms $g$ of $R \ten A$ as a graded $A$-linear $\cP$-algebra with $g \equiv \id \mod \m(A)$, so $\ddef(L,A)$ parametrises isomorphism classes of deformations of $R$. As in \cite{drinfeldtoschechtman,Kon,hinichDefsHtpyAlg,KS}, if $A$, $R$ and $\cP$ are all concentrated in non-negative chain degrees and either $R$ is cofibrant\footnote{If we drop the cofibrant hypotheses, the map from $\ddef(L,A)$ to derived deformations ceases to be an equivalence. If we drop the boundedness hypotheses, we don't even have a map unless we replace $\ten^{\oL}_A$ with a derived tensor product of the second kind.} or we replace $\cP$-derivations with $\cP_{\infty}$-derivations, 
 then $\ddef(L,A)$ is equivalent to the set of all quasi-isomorphism classes of derived deformations of $R$, i.e. $A$-linear $\cP$-algebras $R'$ equipped with a fixed quasi-isomorphism $R'\ten^{\oL}_Ak \simeq R$.
\end{example}

\subsection{Why consider functors on \tps{$dg\Art$}{dgArt}?}

\begin{definition}
 A small extension in $dg\Art_k$ is a surjective morphism $A \to B$ for which the kernel $I$ satisfies $I\cdot \m(A)=0$. We say that the small extension is acyclic if moreover $\H_*I=0$.

 Note that every surjection factorises as a composite of small extensions.
 \end{definition}

Given a functor $F$ on $\Art_k$ and a small extension $A \to B$ in $\Art_k \subset dg\Art_k$, classical obstruction theory is concerned with understanding potential obstructions to lifting elements of $F(B)$ to $F(A)$. Until the advent of derived deformation theory, constructing such obstructions was something of an art, with functoriality often difficult to establish. 

\subsubsection{Extended deformation functors}

\begin{definition}[\cite{Man2}, following \cite{Sch}]\label{edf}
Say that $F: dg\Art_k \to \Set$ is  a  pre-deformation (resp. deformation) functor if:
\begin{enumerate}
\item  for all small extensions $A \onto B$ (equivalently: for all surjections) and all morphisms $C \to B$, the map
$
F(A\by_BC) \to F(A)\by_{F(B)}F(C)
$
is surjective;

\item for all $A,B$, the map
$
F(A\by_kB) \to F(A)\by F(B)
$
is an isomorphism;

\item $F(k)\cong \{\ast\}$;

\item for all acyclic small extensions $A \to B$, the morphism $F(A) \to F(B)$ is a surjection (resp. an isomorphism).
\end{enumerate}

\begin{example}
 Lemma \ref{obsdglalemma} implies that $\ddef(L,-)$
is a deformation functor for any DGLA $L$ and that $\mc(L,-)$ is a pre-deformation functor.
\end{example}
\end{definition}

\begin{definition}\label{HnFdef}
Given a deformation functor $F$, define its $n$th tangent cohomology group by $\H^n(F):= F(k[\eps_n])$; this is in fact a $k$-vector space.

In this notation, Example \ref{ddefcohoex} says that $\H^n(\ddef(L,-))\cong \H^{n+1}(L)$.
\end{definition}


\subsubsection{Obstructions}

The next argument first appeared in  \cite[proof of Theorem 3.1, step 3]{Man2}, and shows that obstruction spaces arise as higher tangent spaces.
\begin{proposition}\label{obsprop}
 For any deformation functor $F \co dg\Art_k \to \Set$ and any small extension $e \co A \to B$ with kernel $I$, we have a natural obstruction map $o_e\co F(B) \to \bigoplus_m \H^{m+1}(F)\ten \H_mI$ whose kernel is the image of $F(A) \to F(B)$.
\end{proposition}
\begin{proof}
Let $\tilde{B}$ be the cone of $e \co I \to A$, regarded as a cdga in the obvious way with  $I\cdot I=0$. Then we have an acyclic small extension $\phi \co \tilde{B} \to B$ and a natural surjection $\rho \co \tilde{B} \onto k \oplus I[1]$ (with $I[1]$ square-zero) such that $A= \tilde{B}\by_{\rho,k \oplus I[1]}k$. 

Since $F$ is a deformation functor, we can then define $o_e$ to be the composite 
\[
 F(B) \xla[\sim]{F(\phi)} F(\tilde{B}) \xra{F(\rho)} F(k \oplus I[1]) \cong  \bigoplus_m \H^{m+1}(F)\ten \H_mI.
\]

Surjectivity of the map 
$F(A) \to F(\tilde{B})\by_{F(k \oplus I[1])}F(k)$ completes the proof. 
\end{proof}

Consequently, a morphism $F \to G$ of deformation functors is an isomorphism if and only if it induces an isomorphism  $\H^*(F) \cong \H^*(G)$ \cite[Corollary 3.3]{Man2}, which in particular implies that quasi-isomorphisms of DGLAs induce isomorphisms of deformation functors. By \cite[Theorem 2.8]{Man2}, every pre-deformation functor $F$ has a universal deformation functor $F^+$ under it, and $\mc(L,-)^+ \cong \ddef(L,-)$  by  \cite[Corollary 3.4]{Man2}.

\section{Koszul duality}\label{koszulsn}


%
%
%

\subsection{Pro-Artinian cdgas, dg coalgebras and the bar/cobar construction}

As in \cite{descent}\footnote{The theory of pro-categories was developed here specifically for its applications to deformation theory.} the pro-category $\pro(dg\Art_k)$ consists of filtered inverse systems $\{A_{\alpha}\}_{\alpha}$ in $dg\Art_k$, with homomorphisms $\Hom(\{A_{\alpha}\}_{\alpha},\{B_{\beta}\}_{\beta}):= \Lim_{\beta}\LLim_{\alpha}\Hom(A_{\alpha},B_{\beta})$. 

The category  $dg\Art_k$ has finite limits and is  Artinian  in the sense that all objects satisfy the descending chain condition on subobjects. By \cite[Corollary to Proposition 3]{descent}, we can thus assume that all morphisms in our inverse systems are surjections and we can contravariantly identify $\pro(dg\Art_k)$ with the category $\lex(dg\Art_k,\Set)$ of left exact (i.e. finite limit-preserving) set-valued functors on $dg\Art_k$. 

Since every vector space is the nested union of its finite-dimensional subspaces, dualisation gives a contravariant equivalence between the category of vector spaces and the category of pro-finite-dimensional vector spaces. Similarly, the functor sending a pro-object $A=\{A_{\alpha}\}_{\alpha}$ to its continuous dual $A^*:= \LLim_{\alpha} A_{\alpha}^*$ defines a contravariant equivalence of categories between $\pro(dg\Art_k)$ and the category of dg coalgebras $C$  which are unital in the sense of \cite[Definition 2.1.1]{hinstack} (i.e. $C= k \oplus \bar{C}$ for $\bar{C}$ ind-conilpotent), because the Fundamental Theorem of Coalgebras ensures that every dg coalgebra is the nested union of its finite-dimensional dg  subcoalgebras. 

The notation in the following definition is fairly nonstandard, with $\Omega \dashv \b$ or $\cL \dashv \C$ being more common than $\beta^* \dashv \beta$.

\begin{definition}
 Define $\beta \co \DGLA_k \to \pro(dg\Art_k)^{\op}$ to be the functor sending a DGLA $L$ to the object pro-representing the left exact functor $\mc(L,-) \co dg\Art_k \to \Set$. Explicitly, $\beta(L)$ is the free local pro-Artinian graded-commutative algebra 
 \[
  \beta(L):=\prod_{n\ge 0} \widehat{\Sym}^n(L^*[-1]) \cong  (\bigoplus_{n \ge 0} \Co\Symm^n(L[1]))^*,
 \]
 with differential $d$ given on generators by the map $L^*[-1] \to L^* \oplus \widehat{\Sym}^2(L^*[-1])[1]$ dual to $d_L  + \half [-,-]$.
 
 Here,  $(-)^*$ denotes  the continuous dual  sending a nested union $\LLim_{\alpha} V_{\alpha}$ of finite-dimensional vector spaces to the pro-finite-dimensional space $\{V_{\alpha}^*\}_{\alpha}$. 
 
 Define $\beta^*$ to be the left adjoint to $\beta$, sending $A$ to the free graded Lie algebra generated by the continuous dual $\m(A)^*[-1]$, with differential given on generators by the map $\m(A)^*[-1] \to \m(A)^* \oplus \Lambda^2(\m(A)^*[-1])[1]$ dual to the sum of $d_A$ and the multiplication map.
 
 In particular, note that 
$
\Hom_{\DGLA_k}(\beta^*A,L) \cong \mc(L,A) \cong \Hom_{\pro(dg\Art_k)}(\beta(L),A).
$
\end{definition}

The following is a rephrasing of \cite[Theorems 3.1 and 3.2]{hinstack} combined with \cite[Proposition 4.36]{ddt1}.
\begin{proposition}\label{quillenequivprop}
 There is a model structure on $\DGLA_k$ in which weak equivalences are quasi-isomorphisms and fibrations are surjections. The adjunction $\beta^* \dashv \beta$ induces a contravariant Quillen equivalence with a model structure on $\pro(dg\Art_k)$ in which fibrations are surjections and weak equivalences are $\beta^*$-quasi-isomorphisms. The latter model structure is fibrantly cogenerated, with cogenerating fibrations (resp. trivial fibrations) given by  small extensions (resp. acyclic small extensions) in $dg\Art_k$.
\end{proposition}
Beware that $\beta^*$-quasi-isomorphism is a more restrictive notion than quasi-isomorphism, although they agree for objects of $\pro(dg_+\Art_k) \subset \pro(dg\Art_k)$ by \cite[Proposition 3.3.2]{hinstack}. The final statement of the proposition implies that the  homotopy category $\Ho(\pro(dg\Art_k))$ is given by localising $\pro(dg\Art_k)$ at limits of filtered systems of acyclic small extensions. This homotopy category is a non-abelian analogue of the derived categories of the second kind we encountered in Example \ref{defchainex}.

\begin{remarks}
 There is nothing particularly special about Lie and commutative algebras in the formation of this adjunction. There are similar constructions for algebras over any Koszul dual pair of dg operads\footnote{This question inspired the notion of Koszul duality for operads, first proposed in \cite{drinfeldtoschechtman}.}, with the Maurer--Cartan functor still providing the adjunction because tensor product of algebras of the respective types yields a DGLA. Such equivalences are explicitly described in a slightly different setting in \cite{CalaqueCamposNuiten}, but also see \S \ref{summarysn}\ref{summarykoszul} below.
 
 In particular, interchanging the roles of commutative and Lie algebras produces the adjunction between Quillen and Sullivan rational homotopy types. The latter have only very limited scope for interaction with derived geometry, since they require  quasi-isomorphism invariance for CDGAs 
 concentrated in non-negative \emph{cochain} degrees.
 
 The full subcategory of $\pro(dg\Art_k)^{\op}$ on fibrant objects is equivalent to the category of $L_{\infty}$-algebras and $L_{\infty}$-morphisms, again via the bar construction.
\end{remarks}

%
%


\subsection{Functors on \tps{$\pro(dg\Art_k)$}{pro(dgArt)}}\label{deffunctorsn}

Although a deformation functor on $dg\Art$ is not necessarily of the form $\ddef(L,-)$ for a DGLA $L$, an analogous statement becomes true if we enlarge our test category to incorporate pro-objects. 

We can extend the functors $\mc(L,-)$ and $\Gg(L,-)$ to the whole of $\pro(dg\Art_k)$ by taking limits, and then set $\ddef(L,A):=\mc(L,A)/\Gg(L,A)$.\footnote{Beware that the map $\ddef(L,\{A_{\alpha}\}_{\alpha}) \to \Lim_{\alpha}\ddef(L,A_{\alpha})$ is seldom an equivalence.} 

For the model structure on $\lex(dg\Art_k,\Set)$ induced by the equivalence with $\pro(dg\Art_k)^{\op}$, a morphism $F \to G$ is a fibration (resp. trivial fibration) if $F(A) \to F(B)\by_{G(B)}G(A)$ is surjective for all acyclic small extensions (resp. all small extensions) $A \to B$. 

\begin{lemma}\label{Ggpath}
 The left exact functor $\Gg(L,-) \by \mc(L,-)$ is a path object in $\lex(dg\Art_k,\Set)$  for the fibrant object $\mc(L,-)$.
\end{lemma}
\begin{proof}
 Lemma \ref{obsdglalemma} implies that $\mc(L,-)$ is fibrant. To show that $\Gg(L,-) \by \mc(L,-)$ is a path object, we need  the map $\omega \mapsto (\id, \omega)$ from $\mc(L,-)$ to be a weak equivalence and  the map $\Gg(L,-) \by \mc(L,-) \to \mc(L,-) \by \mc(L,-)$ sending $(g,\omega)$ to  $(g\star \omega,\omega)$ to be a fibration. 
 
 For the first property, just observe that for any small extension $A \to B$, the map $\Gg(L,A) \to \Gg(L,B)$ is surjective, so the projection map $\Gg(L,-) \by \mc(L,-) \to \mc(L,-) $ is a trivial fibration.
 
 For the second property, we need to show that for any small extension $\phi \co A \to B$ with kernel $I$ and any element $(g,\omega,\omega')$ in $\Gg(L,B)\by \mc(L,A)^2$ such that $g\star \phi(\omega)= \phi(\omega')$, there exists an element $\tilde{g} \in \Gg(L,A)$ lifting $g$ with $\tilde{g}\star \omega = \omega'$.  To see this, lift $g$ to an element $\breve{g} \in \Gg(L,A)$ and note that $\breve{g}\star \omega - \omega' \in \z^1(L\ten I)$. Since $I$ is acyclic, this equals  $dx$ for  some $x \in (L\ten I)^0$, and then setting $\tilde{g}:= \breve{g}-x$ gives the required element.
 \end{proof}

 Since all objects  of $\pro(dg\Art_k)$ are fibrant, as an immediate consequence we have:
\begin{proposition}
For all DGLAs $L$ over $k$ and all $A  \in \pro(dg\Art_k)$, we have natural isomorphisms
 \[
\Hom_{\Ho(\DGLA_k)}(\beta^*A,L) \cong \ddef(L,A) \cong \Hom_{\Ho(\pro(dg\Art_k))}(\beta(L),A).
\]
\end{proposition}

%

\cite[Theorem 6.3]{GuanLazarevShengTangII} then gives the following, by combining Brown-type representability with Proposition \ref{quillenequivprop}:

\begin{theorem}\label{HoDGLAequivthm1}
 The functor $L \leadsto \ddef(L,-)$ gives an equivalence from the category $\Ho(\DGLA_k)$ of $k$-DGLAs localised at quasi-isomorphisms to the category of set-valued functors $F$ on $\pro(dg\Art_k)$ satisfying
 \begin{enumerate}
 \item $F$ sends $\beta^*$-quasi-isomorphisms to isomorphisms, 
  \item for all surjections $A \to B$ and all maps $C \to B$, the map $F(A\by_BC) \to F(A)\by_{F(B)}F(C)$ is a surjection,  and
  \item $F$ preserves  products over $k$ indexed by any (possibly empty) set.
 \end{enumerate}

\end{theorem}

\section{Simplicial functors}
 Because they are defined as quotients, set-valued moduli functors (in particular $\ddef(L,-)$) interact badly with limits, meaning they have poor geometric properties and are seldom representable. Classically, this is resolved by using  groupoid-valued functors such as the Deligne groupoid $\Del(L,-)$, giving rise to moduli stacks in place of moduli spaces. For functors on dg algebras, we have to go further and incorporate homotopies and higher homotopies between automorphisms in order to avoid the same issues. We thus consider functors taking values in $\infty$-groupoids, which are most conveniently modelled as topological spaces (up to weak homotopy equivalence) or simplicial sets (up to Kan--Quillen weak equivalence). 

 Furthermore, functors such as those in Examples \ref{defchainex} and \ref{defalgex} only govern derived deformations when restricted to $dg_+\Art \subset dg\Art$ (cf. \cite{coddt}), but restricting the functor loses information such as negative cohomology groups of the DGLA. By working with simplicial set-valued functors, we can safely restrict to $dg_+\Art$ without losing information. 
 
 \subsection{Hinich's simplicial nerve}
 
 The following, when restricted to $dg_+\Art$, is Hinich's nerve $\Sigma_L$ from \cite[Definition 8.1.1]{hinstack}.
 \begin{definition}\label{mmcdef}
  Given a DGLA $L$, define the simplicial set-valued functor $\mmc(L,-)\co dg\Art \to s\Set$ by 
  \[
   \mmc(L,A)_n:= \mc(\Tot(L\ten \Omega^{\bt}(\Delta^n)),A),
  \]
  with the obvious simplicial structure maps, where $\Omega^{\bt}(\Delta^n)$ is the cdga $\Q[t_0, \ldots, t_n, dt_0, \ldots, dt_n]/(\sum t_i-1,\, \sum dt_i)$ of de Rham polynomial forms on the $n$-simplex, with $t_i$ of degree $0$.
 \end{definition}

 Now, $n \mapsto L \ten \Omega^{\bt}(\Delta^n)$ is a Reedy framing of $L$ in the model category of DGLAs, so the Quillen adjunction $\beta^* \dashv \beta$ gives us weak equivalences
 \[
\oR\map_{\DGLA_k}(\beta^*A,L) \simeq \mmc(L,A) \simeq \oR\map_{\pro(dg\Art_k)}(\beta(L),A)
\]
of simplicial sets, for derived function complexes $\oR\map$ as in \cite[\S 5.4]{hovey}.

\newcommand\DDEL{\uline{\mathfrak{DEL}}}
\begin{definition}\label{ddeldef}
  Given a DGLA $L$, define the simplicial groupoid-valued functor $\DDEL(L,-)\co dg\Art \to \Gpd^{\Delta}$ by 
  \[
   \DDEL(L,A)_n:= \Del(\Tot(L\ten \Omega^{\bt}(\Delta^n)),A),
  \]
 and the functor $\DDel(L,-)\co dg\Art \to s\Gpd$, taking values in simplicially enriched groupoids, by letting $\DDel(L,A)$ have objects $\mc(L,A)$ and simplicial sets $ \DDel(L,A)(\omega,\omega')$ of morphisms given by
 \[
  \DDel(L,A)(\omega,\omega')_n:=\{g \in \Gg(\Tot(L\ten \Omega^{\bt}(\Delta^n))\ten \m(A))~:~ g \star \omega = \omega' \in  \mmc(L,A)_n\}.
 \]
 \end{definition}

 Given a simplicial groupoid $\Gamma$, we can apply the nerve construction $B$ to give a bisimplicial set $B\Gamma$, then take the diagonal to give a simplicial set $\diag B \Gamma$. Then:
 
\begin{lemma}\label{MCDelequivlemma}
There are natural weak equivalences
 \[
  \mmc(L,A) \to \diag B \DDEL(L,A) \la \diag B\DDel(L,A)
 \]
for all DGLAs $L$ and local Artinian cdgas $A$.
\end{lemma}
\begin{proof}
 The simplicial group $\uline{\Gg}(L,A)$ given by $n \mapsto \Gg(L\ten \Omega^{\bt}(\Delta^n),A)$ is contractible. Since the first map is a homotopy quotient by  $\uline{\Gg}(L,A)$, it is a weak equivalence. 
 
The functor  $\DDel(L,-)$ preserves limits, and  an argument similar to Lemma \ref{Ggpath} ensures that it sends small extensions to fibrations of simplicial groupoids, and acyclic small extensions to trivial fibrations. For any simplicial set $X$ and any of the functors $F$ in Definition \ref{ddeldef}, the functor $[X,F(-)]:= \Hom_{\Ho(s\Set)}(X,F(-))$ is thus a deformation functor, so weak equivalence follows by  checking isomorphism on tangent cohomology for spheres, where we have 
 \[
 [S^i, \mmc(L,k[\eps_n])] \cong \H^n\ddef(L\ten \Omega^{\bt}(S^i))  \cong \H^{n+1}(L) \by \H^{n+1-i}(L) \cong [S^i,B\DDel(L,k[\eps_n])].\qedhere 
 \]
 \end{proof}

\begin{example}
As in Example \ref{defalgex}, the DGLA $L:=\DDer_{\cP}(R,R)$ governs deformations of algebras $R$ over $k$-linear dg operads $\cP$. We can now say that $\DDel(L,A)$ is the simplicial groupoid whose objects are $A$-linear deformations of $R$, with  degree $n$ isomorphisms between $\cP\ten A$-algebras $R'$ and $R''$  given by  $\cP\ten A \ten \Omega^{\bt}(\Delta^n)$-algebra morphisms $R' \ten \Omega^{\bt}(\Delta^n) \to R''\ten \Omega^{\bt}(\Delta^n)$ which are the identity modulo $\m(A)$.

 As in \cite{hinichDefsHtpyAlg}, for $A \in dg_+\Art_k$ and $R$ cofibrant, with $\cP$ and $R$ both concentrated in non-negative chain degrees,  the simplicial set $\mmc(L,A) \simeq \diag B\DDel(L,A)$ is weakly equivalent to the $\infty$-groupoid of derived deformations of $R$. Without those constraints, it gives contraderived deformations as in \cite{coddt}.
\end{example}


 \subsection{Simplicial functors on \tps{$dg_+\Art$}{dg+Art}}
 
The following is \cite[Proposition 4.3]{ddt1}: 
 \begin{proposition}\label{quillennonnegmodelprop}
 There is a model structure on  $\pro(dg_+\Art_k)$ in which fibrations are surjective in strictly positive degrees and weak equivalences are quasi-isomorphisms (homology isomorphisms in pro-finite-dimensional vector spaces). The inclusion functor $\pro(dg_+\Art_k) \into \pro(dg\Art_k)$ is left Quillen and preserves weak equivalences; its right adjoint is given by good truncation.
%
\end{proposition}
 
 In contrast to the model structure on the category $\pro(dg\Art_k)$ of unbounded objects from Proposition \ref{quillenequivprop}, for objects in $\pro(dg_+\Art_k)$ the notions of quasi-isomorphism and $\beta^*$-quasi-isomorphism agree \cite[Proposition 3.3.2]{hinstack}. Moreover, every surjective quasi-isomorphism in $\pro(dg_+\Art_k)$ is a transfinite composition of acyclic small extensions \cite[Lemma 4.5]{ddt1}.

We are primarily interested in functors like Hinich's simplicial nerve, so consider the category $\lex(dg_+\Art_k,s\Set)$ of left exact functors from $dg_+\Art$ to simplicial sets. By \cite[Corollary to Proposition 3]{descent}, $\lex(dg_+\Art_k,\Set) $ is equivalent to $\pro(dg_+\Alg_k)^{\op}$, so $\lex(dg_+\Art_k,s\Set) $ is contravariantly equivalent to the category $\pro(dg_+\Alg_k)^{\Delta}$ of cosimplicial diagrams.
The following then  follows from \cite[Proposition 4.12]{ddt1}: 
  \begin{proposition}\label{sdgmodel}
 There is a cofibrantly generated simplicial model structure on $\lex(dg_+\Art_k,s\Set)$ in which a morphism $F \to G$ is:
 \begin{itemize}
  \item
  a fibration if the morphism $F(A) \to G(A)\by_{G(B)}F(B)$ is a Kan fibration (resp. a trivial Kan fibration) for all small extensions (resp. acyclic small extensions) $A \to B$ in $dg_+\Art_k$;
  \item
  a trivial fibration if the  morphism $F(A) \to G(A)\by_{G(B)}F(B)$ is a trivial Kan fibration for all small extensions $A \to B$ in $dg_+\Art_k$.
 \end{itemize}
 For a morphism $\eta \co F \to G$ between fibrant objects, the following conditions are equivalent:
 \begin{enumerate}
  \item $\eta$ is a weak equivalence;
  \item $\eta_A \co F(A) \to G(A)$ is a weak equivalence for all $A \in dg_+\Art_k$;
  \item $\eta_{k[\eps_n]} \co F(k[\eps_n]) \to G(k[\eps_n])$ is a weak equivalence for all $n \ge 0$.
 \end{enumerate}
 \end{proposition}

 \begin{example}
  If $L \to M$ is a surjective morphism of DGLAs, then $\mmc(L,-) \to \mmc(M,-)$ is a fibration. We also have $\pi_i \mmc(L,k[\eps_n])\cong \H^{1+n-i}(L)$, so quasi-isomorphisms of DGLAs give rise to weak equivalences of Hinich nerves. 
 \end{example}

 \begin{remarks}\label{sstrrmk}
  Note that the condition for an object $F$ of $\lex(dg_+\Art_k,s\Set)$ to be fibrant is weaker than asking for the induced map $F \co \pro(dg_+\Art_k) \to s\Set$ to be right Quillen, because small extensions only generate the class of surjections, not all fibrations in $\pro(dg_+\Art_k)$. This slight relaxation introduces groupoid-like behaviour and corresponds to the difference between representability by a scheme or by  a stack.
  
 The simplicial structure on  $\lex(dg_+\Art_k,s\Set)$ is simply given by  defining $F^K(A):= F(A)^K$, for $K \in s\Set$, and then the simplicial set $\HHom(U,F)$ is given by $n \mapsto \Hom(U,F^{\Delta^n})$.
 
 
 In the statement of Proposition \ref{sdgmodel}, we have not described general weak equivalences, but these can be characterised as follows. All objects of the category are cofibrant, so a morphism $U \to V$ is a weak equivalence if and only if the sets $\pi_0\HHom(V,F) \to \pi_0\HHom(U,F)$ of homotopy classes of morphisms are isomorphisms for all fibrant objects $F$.
    
  Proposition \ref{sdgmodel} also has generalisations \cite[Theorem 2.14]{ddt1} to finite and mixed characteristic, using simplicial Artinian rings in place of dg Artinian rings. 
 \end{remarks}

 \section{Representability and comparisons}\label{repcompsn}
 
 \subsection{Representability and Schlessinger's conditions}
 
 Given a functor $F \co dg_+\Art_k \to s\Set$, it is now natural to ask whether it is weakly equivalent to a fibrant object $G$ of $\lex(dg_+\Art_k,s\Set)$, i.e. whether there exists a zigzag of objectwise weak equivalences  going from $F$ to $G$. It is fairly easy to see necessary conditions on $F$, by identifying those properties of $G$ which are invariant under weak equivalence:
 \begin{enumerate}
  \item For any acyclic small extension $A \to B$, the map $F(A) \to F(B)$ is a weak equivalence.
\item For any small extension $A \to B$ and any map $C \to B$ in $dg_+\Art_k$, the map 
\[
F(A\by_BC) \to F(A)\by^h_{F(B)}F(C)
\]
to the homotopy fibre product is a weak equivalence. This follows because $G(A\by_BC) \cong G(A)\by_{G(B)}G(C)$ and $G(A) \to G(B)$ is a Kan fibration, so $G(A)\by_{G(B)}G(C) $ is a model for $G(A)\by_{G(B)}^hG(C)$.
\item Similarly, $F(k)$ is contractible.
  \end{enumerate}
On  taking path components, these recover conditions very close to those of Schlessinger \cite{Sch}, because $\pi_0(X\by^h_YZ) \onto (\pi_0X)\by_{(\pi_0Y)}(\pi_0Z)$ and  $\pi_0(X\by^h Z) \cong \pi_0X \by \pi_0Z$. 

The following is then \cite[Definition 2.28, as adapted in Theorem 4.14]{ddt1}:
  
\begin{definition}\label{schless}
Define the category $\cS$   to consist of functors $F: dg_+\Art_k\to s\Set$ satisfying the  conditions above; we refer to these as derived Schlessinger functors.
%
%
%

Say that a natural transformation $\eta:F \to G$ between such functors is a weak equivalence if the maps $F(A) \to G(A)$ are weak equivalences for all $A\in dg_+\Art_k$, and let $\Ho(\cS)$ be the category obtained by formally inverting all weak equivalences in $\cS$.
\end{definition}

\begin{remark}
 Lurie \cite{lurieDAG10} refers to similar functors on weakly Artinian cdgas as \emph{formal moduli problems}, but the conditions are neither sufficient nor necessary to endow a functor with a natural moduli interpretation, and  Examples \ref{defchainex} and \ref{defalgex} mention examples of natural formal moduli problems (in the traditionally understood sense) which do not give rise to derived Schlessinger functors (i.e. formal moduli problems in Lurie's sense). 
\end{remark}

 Manetti's obstruction theory from Proposition \ref{obsprop} extends to such functors with an almost identical argument:
 \begin{proposition}\label{obsprop2}
 For any derived Schlessinger functor $F \co dg_+\Art_k \to s\Set$ and any small extension $e \co A \to B$ with kernel $I$, we have a natural homotopy fibre sequence
 \[
  F(A) \to F(B) \xra{o_e}F(k \oplus I[1])
 \]
in the homotopy category of simplicial sets, where $I[1]$ is square-zero.
\end{proposition}

Any functor $F \in \cS$ admits a natural extension $\overrightarrow{F}$ first to the pro-category $\pro(dg_+\Art)$ and then to the category of cosimplicial diagrams $\pro(dg_+\Art)^{\Delta}$, in both cases by passing to homotopy limits. Note that such an extension was not available for the set-valued functors of \S \ref{koszulsn}\ref{deffunctorsn} because they tended not to preserve limits. On taking $\pi_0$, we then have a set-valued functor $\pi_0\overrightarrow{F}$ on $\Ho(\pro(dg_+\Art)^{\Delta})$.

\begin{remark}
As for instance in \cite[Definition 1.13]{drep}, we can define tangent cohomology groups\footnote{The notation $\DD$ is based on that for Andr\'e--Quillen cohomology groups in \cite{Q}, with which these coincide for representable functors.} for derived Schlessinger functors by 
$
\DD^{n-i}(F,V):= \pi_iF(k \oplus V[n])
$
for $i,n\ge 0$ and $V$ a finite-dimensional non-negatively graded $k$-chain complex, equipped with zero multiplication.

The derived Schlessinger conditions ensure that this is well-defined and also imply that $\DD^n(F,V) \cong \bigoplus_j  \DD^{n+j}(F,k)\ten \H_jV$. Repeated application of Proposition \ref{obsprop2} implies that tangent cohomology groups detect weak equivalences.
\end{remark}

  The following is then \cite[Theorem 4.14]{ddt1}; although only stated there on the level of homotopy categories, it implies full faithfulness on simplicial localisation  because  the functor preserves the cotensoring of \ref{sstrrmk}, so for
  $K \in s\Set$  we have
  \begin{align*}
   &\Hom_{\Ho(s\Set)}(K, \oR\map_{\lex(dg_+\Art_k,s\Set)_{\mathrm{fib}}}(F,G))  \cong \pi_0\oR\map_{\lex(dg_+\Art_k,s\Set)_{\mathrm{fib}}}(F,G^K) \\
   &\cong \pi_0\oR\map_{\cS}(F,G^K) \cong     \Hom_{\Ho(s\Set)}(K, \oR\map_{\cS}(F, G)).
   \end{align*}
  Essential surjectivity is established there by applying Heller's generalisation 
  \cite{heller} 
  of Brown representability  to $\pi_0\overrightarrow{F}$, along similar lines to Theorem \ref{HoDGLAequivthm1}.  

 \begin{theorem}\label{dgschrep}
The natural functors 
\[
 \lex(dg_+\Art_k,s\Set) \hookleftarrow \lex(dg_+\Art_k,s\Set)_{\mathrm{fib}} \to \cS
\]
 induce equivalences on simplicial localisation at weak equivalences, where $(-)_{\mathrm{fib}}$ denotes the full subcategory of fibrant objects. In particular, there is a canonical equivalence between the  homotopy category $\Ho(\lex(dg_+\Art_k,s\Set))$ and the homotopy category $\Ho(\cS)$.
\end{theorem}

See \cite[Theorem 0.0.13]{lurieDAG10} for a later variant in the setting of  Remark \ref{cfluriermk}, previously stated without proof as \cite[Remark 6.2.5]{lurie}.

\subsection{The equivalences}\label{equivsn}

Our next step is to compare the model categories $ \lex(dg_+\Art_k,s\Set) \simeq (\pro(dg_+\Alg_k)^{\Delta})^{\op}$ and $ \lex(dg\Art_k,\Set) \simeq (\pro(dg\Alg_k))^{\op}$. In order to do so, we will introduce an intermediate category of bigraded algebras, mapping naturally to both. 

\begin{definition}
 Define $DG^+dg_+\Art_k$ to consist of cochain chain complexes $A^{\ge 0}_{\ge 0}$ equipped with a   unital associative bigraded-commutative  multiplication with respect to which the differentials act as derivations, equipped with a homomorphism $A \to k$ for which the kernel $\m(A)$ is nilpotent and finite-dimensional. 
\end{definition}

\begin{definition}
 The total complex $(\Tot A)_n:= \bigoplus_i A^i_{i+n}$ defines a functor $\Tot \co DG^+dg_+\Art_k \to dg\Art_k$, which extends to a functor $\Tot^{\Pi} \co \pro(DG^+dg_+\Art_k) \to \pro(dg\Art_k)$ on passing to limits.
\end{definition}

\begin{definition}
Say that  a  map $f: A \to B$ in $DG^+dg_+\Art_k$ is a small extension if it is surjective   with kernel $I$ satisfying $\m(A)\cdot I=0$. Say that it is an acyclic small extension if moreover $\H_*(\Tot I)=0$.  
\end{definition}

\begin{definition}
 Cosimplicial denormalisation gives a functor $D \co  DG^+dg_+\Art_k \to (dg_+\Art_k)^{\Delta}$, with multiplication given by the Eilenberg--Zilber shuffle product \cite[Definition 4.20]{ddt1}.  
%
\end{definition}

The following is \cite[Theorems 4.26 and 4.48]{ddt1}, and completes the chain of comparisons. It can be motivated by the observation that  generating cofibrations in the model structure on   $\pro(dg_+\Alg_k)^{\Delta}$ from Theorem \ref{sdgmodel} all arise as small extensions in  $(dg_+\Alg_k)^{\Delta}$, with generating trivial cofibrations all also inducing quasi-isomorphisms of product total complexes.  

\begin{theorem}\label{dequiv}
There is a fibrantly cogenerated  model category structure on $\pro(DG^+dg_+\Art_k)$, with  cogenerating fibrations the class of small extensions in $DG^+dg_+\Art_k$ and trivial cogenerating fibrations the class of acyclic small extensions.

Moreover, the functors 
$D:\pro(DG^+dg_+\Art_k) \to \pro(dg_+\Art_k)^{\Delta}$ and $\Tot^{\Pi}:\pro(DG^+dg_+\Art_k) \to \pro(dg\Art_k)$
are right Quillen equivalences.
\end{theorem}

Since $DG^+dg_+\Art_k$ is an Artinian category, we can identify 
$\pro(DG^+dg_+\Art_k)$ with $\lex(DG^+dg_+\Art_k, \Set)$, and then a morphism $F \to G$ is a fibration (resp. trivial fibration) if $F(A) \to F(B)\by_{G(B)}G(A)$ is surjective for all acyclic small extensions (resp. all small extensions) $A \to B$. 

The adjoint functor 
$ 
\Tot^{\Pi}_* \co \lex(dg\Art_k,\Set) \to \lex(DG^+dg_+\Art_k, \Set) 
$
(a right Quillen equivalence) is simply given by $\Tot^{\Pi}_*F(A):= F(\Tot A)$, while the adjoint functor
$
D_* \co \lex(dg_+\Art_k, s\Set)  \to \lex(DG^+dg_+\Art_k, \Set) 
$
(also a right Quillen equivalence) 
is given by the end $D_*F(A):= \int_{n \in \Delta} F_n(D^nA)$. 
 
 
%

 \section{Summary of the argument and generalisations}\label{summarysn}
 
 \subsection{Koszul duality}\label{summarykoszul}
 
 The proofs of the results of \S \ref{koszulsn} rely only on Koszul duality between the Lie and non-unital commutative operads (via the non-unital algebras $\m(A)$ associated to each local Artinian cdga $A \cong k \oplus \m(A)$).
 They generalise to any dg operad $\cL$ (generalising $\Lie$) and any dg co-operad $\C$ (generalising $\Com^*$) equipped with a morphism $\alpha \co \Omega \C \to s\cL$ in the notation of \cite[\S 6.5.9]{lodayvalletteoperads}\footnote{The datum $\alpha$ is known as an operadic twisting morphism, and itself arises as a Maurer--Cartan solution.}
 for which the 
 twisted composite products $\cL\circ_{\alpha}\C$ and $\C\circ_{\alpha}\cL$ of \cite[\S 6.4.11]{lodayvalletteoperads} are acyclic.
 
 
 The map $\alpha$ ensures that for any $\C$-coalgebra $C$ and $\cL$-algebra $L$, we have a set $\mc(\Hom(C,L))$ 
 as in \cite[Definition 11.1.1]{lodayvalletteoperads}.
 Under the additional conditions of \cite[\S 6.6]{lodayvalletteoperads}, the acyclicity conditions  are equivalent to asking that $\alpha$ be a quasi-isomorphism.

 
\subsubsection{The Quillen equivalence}\label{koszulequivsn}
 Writing  $\mathrm{DG}\cL\mathrm{A}$ for  category of $\cL$-algebras in cochain complexes, $\mathrm{DG\C C}^{\nilp}$ for the category of ind-conilpotent  $\C$-coalgebras in cochain complexes, i.e. coalgebras for the comonad $V \mapsto \bigoplus_n \C(n)\ten_{k[S_n]} V^{\ten n}$, and  $dg\Art(\C)$ for  the category of duals of finite-dimensional objects of $\mathrm{DG\C C}^{\nilp}$, by \cite[Theorem 2.1]{valletteHtpyThHtpyAlg} we have a
right Quillen equivalence $\beta$:
 \[
 \xymatrix@R=2.5ex{
  \mathrm{DG}\cL\mathrm{A} \ar[r]^-{\mc} \ar[d]_{\beta} & \lex(dg\Art(\C),\Set)\\
      \mathrm{DG\C C}^{\nilp}   \ar@{-}[r]^-{\sim} & \pro(dg\Art(\C))^{\op} \ar@{-}[u]_-*[@]{\sim}                                            ,
   }
   \]
for the projective model structure on $\mathrm{DG}\cL\mathrm{A}$ (in particular with quasi-isomorphisms as weak equivalences) and a model structure  ``of the second kind'' (by analogy with \cite{positselskiDerivedCategories}) on $\mathrm{DG\C C}^{\nilp}$ for which fibrant objects are those which are cofree as graded coalgebras (forgetting the differential) and  weak equivalences are $\beta^*$-quasi-isomorphisms. Weak equivalences between fibrant objects $C$ are maps inducing quasi-isomorphisms on complexes $\tan(C) \subset C$ of indecomposables, i.e. elements on which all operations vanish 
(isomorphic to cogenerators).

We can go further and say that  generating cofibrations (resp. trivial cofibrations) in $\mathrm{DG\C C}^{\nilp}$ are dual to central (resp. acyclic central) extensions in  $dg\Art(\C)$.  

  \subsubsection{Key features of the proof} 

The $\infty$-equivalence in \cite{valletteHtpyThHtpyAlg} between $\cL$-algebras up to quasi-isomorphism and $\C$-coalgebras up to $\beta^*$-quasi-isomorphism simply follows because the co-unit $ \beta^*\beta L \to L$ of the adjunction is always a quasi-isomorphism, by acyclicity of $\cL\circ_{\alpha}\C$.  A similar argument using the central series filtration shows that for fibrant $C \in  \mathrm{DG\C C}^{\nilp}$, the tangent $ \tan C \to \tan (\beta \beta^*C) \cong \beta^*C[1]  $ of the unit is a quasi-isomorphism.
 
 That 
$\beta^*$-quasi-isomorphisms are generated by duals $f \co C \to D$ of acyclic central extensions. 
is  more subtle. Centrality means the filtration $F_0D:=C$, $F_1D:=D$ on $D$ is compatible with the coalgebra structure, inducing an increasing filtration on $\beta^*D$ with associated graded $\beta^*(C \oplus \coker f)$, graded by powers of the acyclic complex $\coker f$, making $f$ a $\beta^*$-quasi-isomorphism. Conversely, as in the first part of proof of \cite[Theorem 11.3.7]{lodayvalletteoperads}, 
acyclicity of $\C \circ_{\alpha}\cL$ implies the unit  $C \to \beta\beta^*C$ of the adjunction is a transfinite composition  of such maps. 
 
 
As in Theorem \ref{HoDGLAequivthm1}, we can then identify the homotopy category  $\Ho( \mathrm{DG}\cL\mathrm{A})$ with the category of set-valued functors on $(\mathrm{DG\C C}^{\nilp})^{\op} $ satisfying some half-exactness conditions, the functor associated to $L$ being $\ddef(\Hom_k(-,L))$.

 \subsubsection{Generalisations}
 There are also generalisations taking the algebras and coalgebras in larger semisimple symmetric monoidal categories, such as categories of representations of pro-reductive algebraic groups, as feature in pro-algebraic homotopy theory. 
 
 The characteristic $0$ hypothesis for the equivalence of \S \ref{summarysn}\ref{summarykoszul}\ref{koszulequivsn} is only needed because we work with symmetric operads; a similar equivalence exists over any base field for non-symmetric operads such as the associative operad.
 
 \subsection{Relating Manetti's functors with simplicial functors}
 
 The logical next step in the comparison, as covered in \S \ref{repcompsn}\ref{equivsn}, is to discard the negative cochain degrees of $\C$-coalgebras (dually, the negative chain degrees of $\C$-algebras), recovering the same data by working with simplicial functors. For this approach to work in general, we have to assume that $\C$ is concentrated in non-negative cochain degrees.\footnote{There is an obvious generalisation if we allow $\C$ to be the total complex of a co-operad in chain cochain complexes, but the simplicial functors become less manageable because the category of test objects changes with each simplicial degree.} \footnote{However, comparison with weakly Artinian objects as  in Remark \ref{cfluriermk} relies on coconnectivity of the Lie operad to ensure an analogue of \cite[Lemma 4.3.1]{boothDefThPt}, specifically that 
 for all $i$, $\H^{i-n}\cL(n)=0$ for $n \gg 0$.
 }
 
 \subsubsection{Quillen equivalences}
 For such co-operads, we have the following Quillen adjunctions, writing right Quillen equivalences on the top row and left Quillen equivalences on the bottom:
 \[
 \xymatrix@R=2.5ex{
  \lex(dg\Art(\C),\Set) \ar[r]^-{\Tot_*}  &  \lex(DG^+dg_+\Art(\C),\Set)   &  \ar[l]_-{D_*} \lex(dg_+\Art(\C),s\Set)   \\
  \mathrm{DG\C C}^{\nilp} \ar@{-}[u]_-*[@]{\sim} & \ar[l]_{\Tot} \mathrm{dg_+DG^+\C C}^{\nilp} \ar[r]^D \ar@{-}[u]_-*[@]{\sim} & (\mathrm{DG^+\C C}^{\nilp})^{\Delta^{\op}}\ar@{-}[u]_-*[@]{\sim},
  }
 \]
where weak equivalence in these model structures is again stronger than quasi-isomorphism.

Trivial fibrations in $\lex(dg\Art(\C),\Set)$ and $\lex(DG^+dg_+\Art(\C),\Set)$ send central extensions to surjections, while fibrations send acyclic  extensions (resp. $\Tot$-acyclic central extensions) to surjections. Trivial fibrations in $\lex(dg_+\Art(\C),s\Set)$ map central extensions to trivial Kan fibrations, while fibrations map central extensions to Kan fibrations and acyclic central extensions to trivial Kan fibrations.
\footnote{One step in establishing the model structure \cite[Theorem 2.14]{ddt1}  on $\lex(dg_+\Art(\C),s\Set)$ requires a more refined argument outside the commutative setting. Given a central extension $A \to B$ with kernel $M$, and a cofibration $A \to S$, the first unnumbered corrigendum lemma preceding that theorem  uses the description of the pushout as a central extension with kernel $M\ten_AS$. In general, we instead have to consider the filtration $F$ on $S$ generated by $M \subset F^1S$, so  $\gr_F^0S \cong S\coprod_AB$ and $\hat{\gr}_FS \cong (\gr_F^0S)\coprod_B(B \oplus M)$. Then $S$ is the limit of   a sequence of abelian extensions $S/F^{p+1}S \to S/F^pS$, with associated   obstructions in $\H^2\Tot^{\Pi}\DDer(\gr_F^0S, \gr_F^pS)$. When $M$ has a contracting homotopy, it is inherited by $\gr_F^pS$ for $p>0$, so the obstruction vanishes.}

Fibrant objects $C \in \mathrm{dg_+DG^+\C C}^{\nilp}$ are those which are cofree as bigraded ind-conilpotent coalgebras and satisfy the additional condition that the subcomplex $\tan(C) \subset C$ of indecomposables satisfies $\H_i\tan(C)^n=0$ for all $i\ge 0, n\ge 0$ and $\H^n\tan(C)_i=0$ for all $i>0,n>0$, with reasoning similar to \cite[Lemma 1.56]{ddt1}. Morphisms between fibrant objects are weak equivalences precisely when they induce isomorphisms on $\H_*\H^*\tan$.

\subsubsection{Key features of the proof}
That the adjunction $D \dashv D_*$ gives an $\infty$-equivalence follows in two stages. The proof of \cite[Lemma 4.25]{ddt1}, which applies the Dold--Kan and Eilenberg--Zilber comparisons to the central series filtration, adapts generally to show that for fibrant objects $C \in (\mathrm{DG^+\C C}^{\nilp})^{\Delta^{\op}}$, the co-unit $DD_*C \to C$ of the adjunction is a transfinite composition of duals of acyclic central extensions. 
Moreover, as in  \cite[Lemma 4.25]{ddt1}, 
$D$  reflects ($\Tot$-acyclic) central extensions, so it follows that for any $B \in \mathrm{dg_+DG^+\C C}^{\nilp}$ and any fibrant replacement $DB \into \widehat{DB}$,  the adjoint map $B \into D_*\widehat{DB}$ is a trivial cofibration. 

That the adjunction $\Tot \dashv \Tot_*$ gives an $\infty$-equivalence is a more unusual argument \cite[Theorem 4.48]{ddt1}. The left Quillen functor $\Tot$ preserves fibrant objects, since they are cofree,  and satisfies $\tan(\Tot C) \cong \Tot\tan(C)$. Moreover, $\Tot_*$ automatically preserves fibrant objects and has $\H_i\H^0\tan(\Tot_*B) \cong \H^{-i}\tan B$ and $\H_0\H^n\tan(\Tot_*B)\cong \H^n\tan B$ for such objects.
 For any fibrant $C \in  \mathrm{dg_+DG^+\C C}^{\nilp}$ (resp.  fibrant $B \in  \mathrm{DG\C C}^{\nilp}$), the unit  $C \to \Tot_*\Tot C$  (resp. co-unit $\Tot \Tot_*B \to B$) is then a weak equivalence between fibrant objects because it induces a $\Tot$-quasi-isomorphism (resp. quasi-isomorphism) on indecomposables. 

\subsection{Representability}
The comparison is completed  by establishing representability using Heller representability \cite{heller} 
as in Theorem \ref{dgschrep}, with the natural inclusions
\[
\xymatrix@1{\lex(dg_+\Art(\C),s\Set) & \ar@{_{(}->}[l] \ \lex(dg_+\Art(\C),s\Set)_{\mathrm{fib}} \ar@{^{(}->}[r] & \cS( \C)} 
\]
becoming equivalences on simplicial localisation at weak equivalences. Here, $\cS(\C)$ consists of functors from $dg_+\Art(\C)$ to $s\Set$ which map acyclic central extensions to weak equivalences and preserve homotopy fibre products whenever one of the maps is a central extension.

Again, this step holds for any dg co-operad $\C$ concentrated in non-negative chain degrees, though it is also true much more generally. Analogues are given in \cite[Theorems 2.14 and 2.30]{ddt1} for simplicial Artinian rings in finite and mixed characteristic, but the argument should apply to much more general Artinian categories with well-behaved analogues of the classes of central and acyclic central extensions. 

\begin{corollary}\label{maincor}
 The functor $L \mapsto \mmc(L,-)$ from the category of DG $\cL$-algebras to the category $\cS(\C)$ of derived Schlessinger functors induces an equivalence of $\infty$-categories on localisation at weak equivalences. In particular, it induces an equivalence $\Ho(\mathrm{DG}\cL A_k)\simeq \Ho(\cS(\C))$ of homotopy categories. 
 \end{corollary}
 \begin{proof}
  Combining the analogues of Proposition \ref{quillenequivprop}, Theorem \ref{dequiv} and Theorem \ref{dgschrep}, we have equivalences
  \begin{align*}
   \Ho(\mathrm{DG}\cL \mathrm{A}_k)\xra[\sim]{\mc} \Ho(\lex( dg\Art(\C),\Set)) \simeq &\Ho(\lex(DG^+dg_+\Art(\C), \Set) )\\
   &\simeq \Ho(\lex(dg_+\Art(\C), s\Set)) \simeq \Ho(\cS(\C)),
  \end{align*}
and similarly on the corresponding $\infty$-categories. It remains to show that the composite functor has the form claimed, so assume that $F \in  \Ho(\lex(dg_+\Art(\C), s\Set)) $ is a fibrant object corresponding to a DG $\cL$-algebra $L$. 

Let $h_A \in \lex(dg_+\Art(\C),\Set)$ be the object represented by $A \in \pro(dg_+\Art(\C))$, and then since $n \mapsto F^{\Delta^n}$ is a Reedy framing for $F$, we have
\[
 \oR\map_{\lex(dg_+\Art(\C),s\Set)}(h_A,F) \simeq \left(n \mapsto \Hom_{\lex(dg_+\Art(\C),s\Set)}(h_A,F^{\Delta^n})\right) \cong F(A).
\]
Because the  equivalences in Theorem \ref{dequiv} preserve the respective copies of $\pro(dg_+\Art(\C))\simeq \lex(dg_+\Art(\C),\Set)^{\op}$ as subcategories, this must also be equivalent to 
\[
 \oR\map_{\lex(dg\Art(\C),\Set)}(h_A,\mc(L,-)) \simeq \mmc(L,A). \qedhere
\]
 \end{proof}

 
\subsubsection*{Acknowledgements} I would like to thank the anonymous referees for diligently identifying omissions, oversights and particularly gaps in some arguments for generalising results to the stated generality.
%
%

\end{document}